\documentclass[a4paper, reqno, 14pt]{amsart}

\usepackage[usenames,dvipsnames]{color}
\usepackage{amsthm,amsfonts,amssymb,amsmath,amsxtra}
\usepackage[all]{xy}
\SelectTips{cm}{}
\usepackage{xr-hyper}
\usepackage[colorlinks=
citecolor=Black,
linkcolor=Red,
urlcolor=Blue]{hyperref}

\usepackage{verbatim}

\usepackage[margin=1.25in]{geometry}

\usepackage{mathrsfs}

\usepackage{diagbox}
\usepackage{enumerate}
\usepackage[shortlabels]{enumitem}

\usepackage{tikz}

\usepackage{tkz-euclide}

\usepackage{setspace} 

\RequirePackage{xspace}
\RequirePackage{etoolbox}
\RequirePackage{varwidth}
\RequirePackage{enumitem}
\RequirePackage{tensor}
\RequirePackage{mathtools}
\RequirePackage{longtable}
\RequirePackage{multirow}

\usepackage{tikz}
\usepackage{tikz-cd}
\usepackage{csquotes}

\usepackage[url=false,doi=false,isbn=false,sorting=nyt,giveninits=true,maxbibnames=4]{biblatex}
\bibliography{references.bib}

\def\ge{\geqslant}
\def\le{\leqslant}
\def\a{\alpha}
\def\b{\beta}
\def\g{\gamma}

\def\d{\delta}
\def\D{\Delta}

\def\th{\theta}

\def\i{^{-1}}
\def\dw{{\dot w}}
\def\dx{{\dot x}}

\def\<{\langle}
\def\>{\rangle}

\newcommand{\BC}{\ensuremath{\mathbb {C}}\xspace}

\newcommand{{\BG}}{\ensuremath{\mathbb {G}}\xspace}

\newcommand{{\BK}}{\ensuremath{\mathbb {K}}\xspace}

\newcommand{\BR}{\ensuremath{\mathbb {R}}\xspace}
\newcommand{\BS}{\ensuremath{\mathbb {S}}\xspace}

\newcommand{\BZ}{\ensuremath{\mathbb {Z}}\xspace}

\newcommand{\CO}{\ensuremath{\mathcal {O}}\xspace}

\newcommand{\CS}{\ensuremath{\mathcal {S}}\xspace}

\newcommand{\Ad}{{\mathrm{Ad}}}

\newcommand{\GL}{\mathrm{GL}}

\newcommand{\id}{\ensuremath{\mathrm{id}}\xspace}

\def\kk{\mathbf k}

\newtheorem{theorem}{Theorem}
\newtheorem{proposition}[theorem]{Proposition}
\newtheorem{lemma}[theorem]{Lemma}
\newtheorem {conjecture}[theorem]{Conjecture}

\theoremstyle{definition}
\newtheorem{definition}[theorem]{Definition}
\newtheorem{example}[theorem]{Example}
\newtheorem*{example*}{Example}

\newtheorem{remark}[theorem]{Remark}

\newtheorem*{function*}{Function}

\numberwithin{equation}{section}
\numberwithin{theorem}{section}

\setitemize[0]{leftmargin=*,itemsep=\the\smallskipamount}
\setenumerate[0]{leftmargin=*,itemsep=\the\smallskipamount}

\renewcommand{\to}{%
   \ifbool{@display}{\longrightarrow}{\rightarrow}%
   }
\let\shortmapsto\mapsto
\renewcommand{\mapsto}{%
   \ifbool{@display}{\longmapsto}{\shortmapsto}%
   }
\newlength{\olen}
\newlength{\ulen}
\newlength{\xlen}
\newcommand{\xra}[2][]{%
   \ifbool{@display}%
      {\settowidth{\olen}{$\overset{#2}{\longrightarrow}$}%
       \settowidth{\ulen}{$\underset{#1}{\longrightarrow}$}%
       \settowidth{\xlen}{$\xrightarrow[#1]{#2}$}%
       \ifdimgreater{\olen}{\xlen}%
          {\underset{#1}{\overset{#2}{\longrightarrow}}}%
          {\ifdimgreater{\ulen}{\xlen}%
             {\underset{#1}{\overset{#2}{\longrightarrow}}}
             {\xrightarrow[#1]{#2}}}}%
      {\xrightarrow[#1]{#2}}
   }
\makeatother
\newcommand{\xyra}[2][]{%
   \settowidth{\xlen}{$\xrightarrow[#1]{#2}$}%
   \ifbool{@display}%
      {\settowidth{\olen}{$\overset{#2}{\longrightarrow}$}%
       \settowidth{\ulen}{$\underset{#1}{\longrightarrow}$}%
       \ifdimgreater{\olen}{\xlen}%
          {\mathrel{\xymatrix@M=.12ex@C=3.2ex{\ar[r]^-{#2}_-{#1} &}}}%
          {\ifdimgreater{\ulen}{\xlen}%
             {\mathrel{\xymatrix@M=.12ex@C=3.2ex{\ar[r]^-{#2}_-{#1} &}}}
             {\mathrel{\xymatrix@M=.12ex@C=\the\xlen{\ar[r]^-{#2}_-{#1} &}}}}}%
      {\mathrel{\xymatrix@M=.12ex@C=\the\xlen{\ar[r]^-{#2}_-{#1} &}}}%
   }
\makeatletter
\newcommand{\xla}[2][]{%
   \ifbool{@display}%
      {\settowidth{\olen}{$\overset{#2}{\longleftarrow}$}%
       \settowidth{\ulen}{$\underset{#1}{\longleftarrow}$}%
       \settowidth{\xlen}{$\xleftarrow[#1]{#2}$}%
       \ifdimgreater{\olen}{\xlen}%
          {\underset{#1}{\overset{#2}{\longleftarrow}}}%
          {\ifdimgreater{\ulen}{\xlen}%
             {\underset{#1}{\overset{#2}{\longleftarrow}}}
             {\xleftarrow[#1]{#2}}}}%
      {\xleftarrow[#1]{#2}}
   }
\newcommand{\isoarrow}{%
   \ifbool{@display}{\overset{\sim}{\longrightarrow}}{\xrightarrow\sim}%
   }

\begin{document}

\title[]{Convex elements and Steinberg's cross-sections} 

\author[Sian Nie]{Sian Nie}
\address{Academy of Mathematics and Systems Science, Chinese Academy of Sciences, Beijing 100190, China}

\address{ School of Mathematical Sciences, University of Chinese Academy of Sciences, Chinese Academy of Sciences, Beijing 100049, China}
\email{niesian@amss.ac.cn}

\author[Panjun Tan]{Panjun Tan}
\address{Academy of Mathematics and Systems Science, Chinese Academy of Sciences, Beijing 100190, China}
\email{tanpanjun@amss.ac.cn}

\author[Qingchao Yu]{Qingchao Yu}
\address{Institute for Advanced Study, Shenzhen University, Shenzhen, Guangdong, China}
\email{qingchao\_yu@outlook.com}

\thanks{}

\keywords{Steinberg's cross-section, algebraic group}
\subjclass[2010]{20G99, 20F55}


\begin{abstract}
In this paper, we study convex elements in a (twisted) Weyl group introduced by Ivanov and the first named author. We show that each conjugacy class of the twisted Weyl group contains a convex element, and moreover, the Steinberg cross-sections exist for all convex elements. This result strictly enlarges the cases of Steinberg cross-sections from a new perspective, and will play an essential role in the study of higher Deligne-Lusztig representations.  
\end{abstract}

\maketitle


\section*{Introduction}

\subsection{Background} \label{subsec:background}
Let $G$ be a reductive group over an algebraically closed field $\kk$. Fix a maximal torus $T \subseteq G$. Let $B, B^- \supseteq T$ be two opposite Borel subgroups whose unipotent radicals are denoted by $U$ and $U^-$ respectively. Let $N_T$ be the normalizer of $T$ in $G$ and let $W = N_T / T$ be the Weyl group of $G$. The choice of $U$ assigns a length function on $W$ in the usual way.




Let $w \in W$ be an elliptic element, that is, $w$ does not belong to any proper parabolic subgroup of $W$. We fix a lift $\dw \in N_T$ of $w$ and consider the following morphism
\begin{align*}
\Xi_\dw: U \times \dw (U \cap \dw\i U^- \dw) &\longrightarrow U \dw U \\ (y, z) &\mapsto y z y\i. \end{align*}
If $\Xi_\dw$ and $\Xi_{\dw\i}$ are both isomorphisms, the subvariety $\CS_{\dw} := \dw(U \cap \dw\i U^- \dw)$ intersects each conjugacy class of $G$ it meets transversely (see Proposition \ref{prop:transversality}). In this case, the subvariety $\CS_{\dw}$ is referred to as a Steinberg cross-section.  

The notion of Steinberg cross-sections originates in a classical result by Steinberg \cite{St}. It states that if $w$ is a minimal length Coxeter element, then $S_{\dw}$ intersects all regular $G$-conjugacy classes transversely. In \cite{HL}, He and Lusztig extended Steinberg's result in the sense that for any elliptic minimal length element $w$, $S_{\dw}$ is a Steinberg cross-section. By adapting the methods of \cite{HL}, Malten \cite{M} and Duan \cite{Duan} generalized Steinberg cross-sections to all good position elements introduced in \cite{HN12}, dropping the elliptic restriction. When $\kk = \BC$ is the field of complex numbers, Sevosyanov \cite{S} obtained a similar result using a different method.

Steinberg cross-sections are used to construct transversal slices for unipotent classes of $G$ in \cite{S}, \cite{HL} and \cite{Duan}. They also play a crucial role in the study of Lusztig varieties \cite{Lus11-A}, \cite{Lus11-B},\cite{Duan2}, the study of variations of Deligne-Lusztig varieties \cite{He14}, \cite{Chan20}, \cite{CI}, \cite{I}, \cite{Nie2023}, \cite{IN24} and the study of physical rigidity of Kloosterman connections \cite{HJ}.

\subsection{Main result}
The goal of this paper is to extend Steinberg cross-sections to a new class of elements in $W$, which we call convex elements (see Definition \ref{def:convex}). The notion of convex elements is introduced by Ivanov and the first named author, which plays an important role in the study of higher Deligne-Lusztig representations \cite{IN}.

To each convex element $w \in W$ we associate a subvariety $\CS_{\dw}$ of $G$ and a morphism $\Xi_\dw$ (see \S\ref{subsec:cross-section}), which extend the corresponding constructions mentioned in \S\ref{subsec:background} for elliptic elements. The main results of this paper is the following.
\begin{theorem} \label{main}
\begin{enumerate}
    The following statements are true.

    \item If $w \in W$ is a convex element, then $\Xi_\dw$ and $\Xi_{\dw\i}$ are both isomorphisms and $\CS_\dw$ intersects each conjugacy classes it meets transversely, namely, $\CS_\dw$ is a Steinberg cross-section;
    
    \item All good position elements of $W$ are convex elements. In particular, each conjugacy class of $W$ contains a convex representative.
\end{enumerate}
\end{theorem}
We refer to Theorem \ref{thm:main}, Proposition \ref{prop:transversality} and Theorem \ref{thm:constr} for general statements on possibly disconnected reductive groups. In particular, this extends a main result \cite[Theorem 1.2]{Duan} of Duan on good position element to convex elements. In follow-up work \cite{IN}, Theorem \ref{main} is used in an essential way to decompose higher Deligne-Lusztig representations into irreducible summands. 

To prove the first statement of Theorem \ref{main}, the traditional approach in \cite{HL}, \cite{M} and \cite{Duan} is to use good elements of $W$ introduced by Geck and Michel \cite{GM97}, which are defined in terms of Deligne-Garside normal forms in the braid monoid of $W$. Here we take a conceptually new perspective. By definition, the convex condition on $w$ assigns a natural filtration of subgroups of $U$, which is compatible with the adjoint action of $\dw$. By keeping track of this filtration, it follows straightforwardly that $\Xi_\dw$ is an isomorphism. In particular, this gives a new and much simpler proof for classical results on Steinberg cross-sections. The second statement is proved by a geometric interpretation of the convex condition for good position elements inspired from \cite{HN12}. In Example \ref{ex:234123}, we show that there do exist convex elements which are not good position elements. Hence our result discovers many new cases of Steinberg cross-sections. Finally, we conjecture that all the minimal length twisted Coxeter elements are convex, see Conjecture \ref{conj:cox}.

\subsection{Structure of the paper} 
The paper is organized as follows. In \S \ref{sec:convex}, we introduce the definition of convex elements and give a short proof for the corresponding existence result of Steinberg cross-sections. In \S\ref{sec:existence}, we provide an inductive construction of convex elements in each conjugacy class of $W$. In \S\ref{sec:good} we recall the notion of good position elements, and show it is strictly stronger than the notion of convex elements. As an application, we reprove classical results by He-Lusztig and Duan on Steinberg cross-sections. In the last section, we propose the conjecture that each minimal length twisted Coxeter element is convex. We show this conjecture is true in a special case.

\subsection*{Acknowledgment}
We would like to thank Xuhua He for helpful comments and suggestions.

\section{Convex Elements} \label{sec:convex}
\subsection{Convex elements}\label{sec:1.1}
Let $G$, $T$, $N_T$, $W$, $B= T U$, $B^- = T U^-$ be as in the introduction. Let $\Phi$ be the root system of $T$ in $G$. The Borel subgroup $B$ determines a set $\Phi^+$ of positive roots. Then $\Phi = \Phi^+ \sqcup \Phi^-$, where $\Phi^- = - \Phi^+$. Let $\D$ be the set of simple roots in $\Phi^+$, and $\BS$ the corresponding set of simple reflections. The pair $(W, \BS)$ is a Coxeter system, equipped with a length function $\ell: W \to \BZ$. For any $w\in W$, we fix a lift $\dw$ in $N_T$.

For any $\a \in \Phi$, we denote by $U_\a \subseteq G$ the corresponding root subgroup. Let $R$ be a subset of $\Phi$ such that $R \cap -R = \emptyset$. We say $R$ is closed if $\a + \b \in R$ for any $\a, \b \in R$ with $\a + \b \in \Phi$. In this case, we define $U_R = \prod_{\a \in R} U_\a$, which is a unipotent subgroup of $G$. Note that the product is independent of the order.

Let $\d$ be an automorphism of $G$ preserving $T$ and $B$. Then $\d$ induces automorphisms of $(W, \BS)$, $\Phi$, $\Phi^+$ and $\D$ in a natural way. By abuse of notation, we still denote these automorphisms by $\d$.

Let $x = w \d^k \in W \rtimes \<\d\>$ with $w \in W$ and $k \in \BZ$. We set $\dx = \dw \d^k \in G \rtimes \<\d\>$. Then $\dx U_\a \dx\i = U_{x(\a)}$. Define 
\begin{gather*}
\Phi^\pm(x) = \{\g\in\Phi^\pm ;  x^i(\g) \in \Phi^\pm \text{ for all } i \in \BZ\};\\
\Phi(x) = \Phi^+(x) \sqcup \Phi^-(x).
\end{gather*}
Then $\Phi(x)=\Phi(x\i)$ and $x(\Phi(x)) = \Phi(x)$. For any $\g \in \Phi^\pm \setminus \Phi(x)$, define 
\begin{align*}
n_x(\g) = \min\{i \in \BZ_{\ge 1}; x^i(\g) \in \Phi^\mp\}\in \BZ_{\ge1}.
\end{align*}
We set $\Phi_{x, i}^\pm = \{\g \in \Phi^\pm; n_x(\g) = i\}$ and $\Phi_{x, \le i}^\pm = \{\g \in \Phi^\pm; n_x(\g) \le i\}$ for each $i \in \BZ_{\ge 0}$. By definition, we can easily see that $n_x(\a + \b) \ge \min\{n_x(\a), n_x(\b)\}$ for any $\a,\b\in\Phi^+$ with $\a+\b\in \Phi^+$.

\begin{definition}\label{def:convex}
    We say $x \in W \rtimes \<\d\>$ is quasi-convex (for $\Phi^+$) if the following two conditions hold: 
    \begin{enumerate}
    \item $\Phi(x) \subseteq \Phi$ is a standard parabolic root subsystem, i.e., $\Phi(x) = \Phi \cap \BZ J$ for some $J \subseteq \D$; 
    
    \item $n_x(\a + \b) \le \max\{n_x(\a), n_x(\b)\}$ for any $\a,\b\in\Phi^+$ with $\a+\b\in \Phi^+$.
\end{enumerate} 
In this case, we define $T(x)$ to be the subgroup of $T$ generated by $T^x:=\{t\in T; \dx t \dx\i =t\}$ and $\{\a^{\vee}(t); \a\in \Phi(x),t\in T\}$ and define $L_x = U_{\Phi^+(x)} T(x) U_{\Phi^-(x)} \subseteq G$.
\end{definition}

We point out that the condition (2) above is equivalent to the following:

\noindent(2') $n_x(\a + \b) \le n_x(\b)$ for any $\a \in \Phi^+ \cap x\i(\Phi^-)$ and $\b \in\Phi^+$ such that $\a+\b\in \Phi^+$.

Indeed, suppose $n_x(\a)\ge2$ and $n_x(\b)\ge2$. Let $j = \min\{n_x(\a),n_x(\b)\}-1$. Then $n_x(\a+\b)\le\max\{n_x(\a),n_x(\b)\}$ if and only if $n_x(x^j(\a)+x^j(\b))\le\max\{n_x(x^j(\a)),n_x(x^j(\b))\}$.

\begin{definition}
    We say $x \in W \rtimes \<\d\>$ is convex (for $\Phi^+$) if both $x$ and $x\i$ are quasi-convex.
\end{definition}

The following example shows that the condition of being convex is stronger than that of being quasi-convex.
\begin{example}
In the case of type $A_4$, the element $x=s_1s_2s_3s_4s_1s_2$ is quasi-convex but $x\i$ is not. Indeed, $n_{x\i}(\a_{23})=1$, $n_{x\i}(\a_{35})=2$ but $n_{x\i}(\a_{25})=3$. 
\end{example}

\begin{lemma} \label{basic}
    Let $x \in W \rtimes \<\d\>$ be a quasi-convex element. Then for any $i \in \BZ_{\ge 0}$, we have
    \begin{enumerate}
        \item $\a + \b \in \Phi_{x, i}^\pm$ for any $\a \in \Phi(x)$ and $\b \in \Phi_{x, i}^\pm$ such that $\a + \b \in \Phi$;
    
        \item $\Phi_{x, i}^\pm$ and $\Phi_{x, \le i}^\pm$ are closed, and the subgroups $U_{\Phi_{x, i}}$ and $U_{\Phi_{x, \le i}}$ are normalized by $L_x$.
    \end{enumerate}
\end{lemma}
\begin{proof}
    (1) follows from the assumption that $\Phi(x) \subseteq \Phi$ is a standard parabolic root subsystem stabilized by $x$. (2) follows from the inequality $\min\{n_x(\a), n_x(\b)\} \le n_x(\a + \b) \le \max\{n_x(\a), n_x(\b)\}$ for any $\a,\b\in\Phi^+$ with $\a+\b\in \Phi^+$.
\end{proof}
Lemma \ref{basic} implies that each quasi-convex element $x$ assigns a filtration of subgroups \[\{1\} \subseteq U_{\Phi_{x,\le1}^+}\subseteq U_{\Phi_{x,\le2}^+} \subseteq \cdots \subseteq U.\] This filtration will be used in the proof of Theorem \ref{thm:main}.
\subsection{The isomorphism $\Xi_{\dx}$} \label{subsec:cross-section}
Let $x \in W \rtimes \<\d\>$ be a quasi-convex element. Note that $\dx L_x = L_x \dx$. Define
\begin{align*}
\CS_{\dx} = \dx L_x U_{\Phi_{x,1}^+} = \dx L_x U_{\Phi^+ \cap x\i(\Phi^-)} \subseteq G\rtimes\<\d\>.
\end{align*}
For a different choice of the lifting $\dot{x}$, $\CS_{\dot{x}}$ differs by conjugation by some element in $T$.

Our definition of $\CS_\dx$ for convex elements $x$ extends that for good position elements in \cite{Duan}. See Remark \ref{rmk:compare} (1) and Remark \ref{rmk:non-convex}.


The main result of this section is the following. 
\begin{theorem}\label{thm:main}
    Let $x\in W \rtimes \<\d\>$ be a quasi-convex element. Then the map $(y, z) \mapsto y z  y\i$ gives an isomorphism 
    $$\Xi_{\dx}: U_{\Phi^+ \setminus \Phi(x)} \times \dx L_x U_{\Phi^+ \cap x\i(\Phi^-)}  \overset \sim \longrightarrow U_{\Phi^+ \setminus \Phi(x)} \dx  L_x U_{\Phi^+ \setminus \Phi(x)}.$$
\end{theorem}

\begin{proof}
   For simplicity, we write $L = L_x$, $U_i = U_{\Phi_{x, i}^+}$ and $U_{\le i} = U_{\Phi_{x, \le i}^+}$ for each $i \in \BZ_{\ge 0}$. Note that $U_0 = \{1\}$, $U_{\Phi^+ \cap x\i(\Phi^-)} = U_1$ and $U_{\Phi^+ \setminus \Phi(x)} = U_{\le N}$, where $N = \max\{n_x(\g); \g\in\Phi^+ \setminus \Phi(x)\}$.

   We now construct a section $\Sigma_{\dx}$ of $\Xi_{\dx}$. Note that there is a natural isomorphism of varieties $r: U_{\le N} \times \dx U_1 L \longrightarrow U_{\le N} \dx U_{\le N} L $ given by multiplication. Define
   \[\phi: U_{\le N} \dx U_{\le N} L   \xlongrightarrow{r^{-1}}  U_{\le N} \times \dx U_1 L \to U_{\le N} \times \dx U_{\le N} L, \]
   where the second morphism is given by $ (y, z) \longmapsto (y, z y)$. For $i = 2, 3, \dots, N$ we define \[\varphi_i: U_{\le N} \times \dx U_{\le i}L  \to U_{\le N} \times \dx U_{\le i-1}L\] as follows. Let $(y, z) \in U_{\le N} \times \dx U_{\le i}L $. We can write $z = \dx u u' m$ uniquely with $u \in U_i$, $u' \in U_{\le i-1} $ and $m \in L$. Let $u'' = \dx u \dx\i \in  U_{i-1}$. Then $z = u'' \dx u' m $. Since $L$ normalizes $U_{i-1}$, we have \[{u''}\i z u'' = \dx u' m u'' \in \dx U_{\le i-1} L.\] We define $\varphi_i(y, z) = (y u'', {u''}\i z u'') \in U_{\le N} \times \dx U_{\le i-1} L$. Put \[\Sigma_{\dx} = \varphi_2 \circ \varphi_3 \circ \cdots \varphi_{N} \circ  \phi.\] It is easy to see that $\Xi_{\dx} \circ\Sigma_{\dx}$ is the identity map on $U_{\le N} \dx U_{\le N} L$.
   
   It remains to show the injectivity of $\Xi_{\dx}$. Assume that there exist $y \in U_{\le N}$ and $z_1, z_2 \in U_1  L $ such that $y \dx z_1 y\i = \dx z_2$. Suppose that $y \neq 1$. There exists $i \in \BZ_{\ge 1}$ such that $y \in U_{\le i }  \setminus U_{\le i-1} $. Hence $\dx\i y \dx \in U_{\le i+1} \setminus U_{\le i}$. Thus $y \dx z_1 \in \dx (\dx\i y \dx) z_1 \subseteq \dx (U_{\le i+1} \setminus U_{\le i}) L$. On the other hand, $\dx z_2 y \in \dx U_1 L U_{\le i} = \dx U_{\le i} L$. This contradicts that $y \dx z_1 = \dx z_2 y$. Thus $y = 1$ and $z_1 = z_2$ as desired.
\end{proof}

\subsection{Transversality} 
Now we show the transversality property of $\CS_{\dx}$ for convex elements $x$.
\begin{proposition} \label{prop:transversality}
Let $x \in W \rtimes \<\d\>$ be a convex element. Then the subvariety $\CS_{\dx} = \dx L_x U_{\Phi^+\cap x^{-1}(\Phi^-)}$ intersects any $G$-conjugacy classes of $G \dx$ it meets transversely.
\end{proposition}

\begin{proof} Our proof is similar to \cite[Proof of Theorem 1.2 (2)]{Duan}. Let $\mathfrak{g}$, $\mathfrak{t}$, $\mathfrak{t}^x$ and $\mathfrak{t}(x)$ be the Lie algebras of $G$, $T$, $T^x$ and $T(x)$ respectively (cf. Definition \ref{def:convex}). For any closed subset $R \subseteq \Phi$, denote by $\mathfrak{n}_R$ the subalgebra of $\mathfrak{g}$ spanned by root subalgebra corresponding to roots in $R$. Let $\mathfrak{l} = \mathfrak{n}_{\Phi^+(x)}  + \mathfrak{t}(x) + \mathfrak{n}_{\Phi^-(x)}$. Let $g\in \CS_{\dx}$. Identify $\mathfrak{g}$ with the tangent space of $G\dx$ at $g$. It suffices to prove that
\begin{align*}\tag{1.1}\label{eq:1.1}
(\Ad(g^{-1}) - 1))(\mathfrak{g}) + \mathfrak{l} + \mathfrak{n}_{\Phi^+\cap x^{-1}(\Phi^-)} = \mathfrak{g}.
\end{align*}

Since $x$ is quasi-convex, by Theorem \ref{thm:main}, we have the isomorphisms 
\begin{align*}
\Xi_{\dx}: U_{\Phi^+ \setminus \Phi(x)} \times \dx L_x U_{\Phi^+ \cap x^{-1}(\Phi^-)} \overset \sim \to U_{\Phi^+ \setminus \Phi(x)} \dx L_x U_{\Phi^+ \setminus \Phi(x)},\quad (y,z)\mapsto y z y\i.
\end{align*}
Passing to the tangent space at $g$, we get 
\begin{align*}
    (\Ad(g^{-1}) - 1)(\mathfrak{n}_{\Phi^+\setminus\Phi(x)}) + \mathfrak{l} + \mathfrak{n}_{\Phi^+\cap x^{-1}(\Phi^-)} \supseteq \mathfrak{l} + \mathfrak{n}_{\Phi^+\setminus\Phi(x)}.
\end{align*}
Hence, the left hand side of (\ref{eq:1.1}) contains $\mathfrak{l}$ and $\mathfrak{n}_{\Phi^+\setminus\Phi(x)}$. 

Note that $(\dx)\i$ is a lift of $x\i$, $\Phi(x) = \Phi(x\i)$ and $L_x = L_{x\i}$. Since $x\i$ is also quasi-convex, using Theorem \ref{thm:main} again, we get the isomorphism
\begin{align*}
\Xi_{(\dx)\i}: U_{\Phi^+ \setminus \Phi(x)} \times (\dx)\i L_x U_{\Phi^+ \cap x(\Phi^-)} \overset \sim \to U_{\Phi^+ \setminus \Phi(x)} (\dx)\i L_x U_{\Phi^+ \setminus \Phi(x)},\quad (y,z)\mapsto y z y\i.
\end{align*}
Taking the inverses of both sides, it becomes
\begin{align*}\tag{1.2}\label{eq:1.2}
U_{\Phi^+ \setminus \Phi(x)} \times U_{\Phi^+ \cap x(\Phi^-)}L_x' \dx  \overset \sim \to U_{\Phi^+ \setminus \Phi(x)} L_x ' \dx  U_{\Phi^+ \setminus \Phi(x)},\quad (y,z)\mapsto y z y\i,
\end{align*}
where $L_x' = U_{\Phi^-(x)} T(x) U_{\Phi^+(x)} = L_x\i$.

Let $\iota$ be an automorphism on $G$ such that $\iota(T) = T$, the induced action on the root system is $-1$ and the induced action on $W$ is trivial (such an automorphism is called an opposition automorphism in \cite[\S2.8]{LS79}). Note that for any $t\in T$, $\Ad(t)\circ\iota$ also satisfies the above properties. Hence we may assume that $\iota(L_x'\dx) = L_x\dx$. Apply $\iota$ to \ref{eq:1.2}. We obtain
\begin{align*}\tag{1.3}\label{eq:1.3}
U_{\Phi^- \setminus \Phi(x)} \times  U_{\Phi^- \cap x(\Phi^+)} L_x \dx  \overset \sim \to U_{\Phi^- \setminus \Phi(x)} L_x \dx U_{\Phi^- \setminus \Phi(x)},\quad (y,z)\longmapsto y z y\i.
\end{align*}
It is easy to see that $U_{\Phi^- \cap x(\Phi^+)} L_x \dx = \dx L_x U_{\Phi^+ \cap x\i(\Phi^-)}$ and that
\begin{align*}
    U_{\Phi^- \setminus \Phi(x)}  L_x\dx   U_{\Phi^- \setminus \Phi(x)} = U_{\Phi^- \setminus \Phi(x)}  \dx  L_{x} U_{\Phi^+\cap x\i(\Phi^-)} U_{\Phi^- \setminus \Phi(x)}.
\end{align*}
Passing to the tangent space at $g$, (\ref{eq:1.3}) implies 
\begin{align*}
    (\Ad(g^{-1}) - 1)(\mathfrak{n}_{\Phi^-\setminus\Phi(x)}) + \mathfrak{l} + \mathfrak{n}_{\Phi^+\cap x^{-1}(\Phi^-)} \supseteq \mathfrak{l} + \mathfrak{n}_{\Phi^-\setminus\Phi(x)}.
\end{align*}
Hence, the left hand side of (\ref{eq:1.1}) contains $\mathfrak{n}_{\Phi^-\setminus\Phi(x)}$.

Now it suffices to prove that the left hand side of (\ref{eq:1.1}) contains the complement subspace $\mathfrak{t}_x := (\Ad(x^{-1}) - 1)(\mathfrak{t})$ of $\mathfrak{t}^x$ in $\mathfrak{t}$. Since $g \in \dx L_x U_{\Phi^+\cap x^{-1}(\Phi^-)}$, we have
\begin{align*}
     (\Ad(x^{-1}) - 1)(\mathfrak{t}) \subseteq (\Ad(g^{-1}) - 1)(\mathfrak{t}) + \mathfrak{l} + \mathfrak{n}_{\Phi^+ \cap x\i(\Phi^-)}.
\end{align*}
Therefore, $\mathfrak{t}_x$ (and hence $\mathfrak{t}$) is contained in the left hand side of (\ref{eq:1.1}). 
\end{proof}

\section{Inductive Construction of Convex Elements}\label{sec:existence}
In this section, we show the existence of convex elements in each $W$-conjugacy class of $W \rtimes \<\d\>$.

Let $V = \BR \Phi$, which is a Euclidean space with an inner product $( - , - )$ preserved by $W \rtimes \<\d\>$. Let $C_0 =\{v \in V; (v,\a) >0 \text{ for any } \a \in \Phi^+\}$ be the dominant Weyl chamber. For any $\g \in \Phi$, denote by $H_{\g} \subseteq V$ the root hyperplane corresponding to $\g$. For a root subsystem $\Psi \subseteq \Phi$ and a linear subspace $K \subseteq V$, we say a point $e \in K$ is $\Psi$-regular in $K$, if for any $\g\in \Psi$, $e \in H_{\g}$ implies that $K \subseteq H_\g$. For any $e\in V$, we set $\Phi_e = \{\a\in\Phi;\<e,\a\> = 0 \}$ and $W_e = \{w\in W;w(e) = e \}$. 

For each $x \in W \rtimes \<\d\>$, there is an orthogonal decomposition \[V = \bigoplus_{0 \le \th \le \pi} V_x^\th,\] where $V_x^\th = \{v \in V;  x(v) + x\i(v) = 2\cos\th \cdot v\}$.

The following key lemma is inspired from \cite[Lemma 2.1]{HN12}. 
\begin{lemma} \label{lem:rotation} 
Let $x \in W \rtimes \<\d\>$ and $0 < \th \le \pi$. Assume that $\overline{C_0}$ contains a $\Phi$-regular point of $V_x^\th$. Let $\Psi = \{\g \in \Phi; V_x^\th \subseteq H_\g\}$ and $\Psi^+ = \Psi \cap \Phi^+$. Then we have
\begin{enumerate}
    \item $x(\Psi) =\Psi$ and $\Psi$ is a standard parabolic root subsystem of $\Phi$;
    \item $\Phi(x) \subseteq \Psi$;
    \item $n_x(\a + \b) \le \max\{n_x(\a), n_x(\b)\}$ for any $\a, \b \in \Phi^+ \setminus \Psi$ such that $\a+\b \in \Phi^+$;
    \item if $x$ is convex for $\Psi^+$, then it is convex for $\Phi^+$.
\end{enumerate} 
Recall that $\Phi(x) = \{\pm\g ; \g\in \Phi^+ \text{ such that } x^i(\g) \in \Phi^+ \text{ for all } i \in \BZ\}$.
\end{lemma}
\begin{proof}
    By definition, $x(V_x^\th) = V_x^\th$. Hence $x(\Psi) = \Psi$. Let $e \in \overline{C_0}$ be a $\Phi$-regular point of $V_x^\th$. Then $\Psi = \Phi_e = \<\a\in \Phi ; \<e, \a\> = 0\>$ and hence is a standard parabolic root subsystem. Thus (1) is proved.

    Let $U$ be the linear subspace of $V$ spanned by $x^i(e)$ for all $i \in \BZ$. Let $\g \in \Phi^+\setminus \Psi$. Set $L_\g = H_{\g} \cap U$. Then $L_\g$ is a hyperplane of $U$. Since $e$ is a regular point in $V_x^\th = x(V_x^\th)$, $(e,x^{i}(\g)) \neq 0$ for all $i\in\BZ$. We claim that 

    (a) $n_x(\g)$ equals the minimal integer $i \ge 1$ such that $e, x^{-i}(e)$ are separated by $L_\g$ (such $i$ exists since $0 < \th \le \pi$).

    Indeed, as $e \in \overline{C_0}$, for any $i\in\BZ$, $x^{i}(\g)\in\Phi^- $ if and only if $(x^{-i}(e),\g) = (e, x^i(\g)) <0$, that is, $e$ and $x^{-i}(e)$ are separated by $L_\g$. Hence (a) follows. In particular, $\g\notin \Phi(x)$. So (2) is proved.
    
    We now prove (3). Let $\a, \b \in \Phi^+ \setminus \Psi$ be such that $\a + \b \in \Phi^+$. First assume that $L_\a = L_\b$, then $L_{\a + \b} = L_\a$. By (a), we have $n_x(\a + \b) = n_x(\b) \le \max\{n_x(\a), n_x(\b)\}$. Now assume that $L_\a \neq L_\b$. If $\th = \pi$, then $x(e) = -e$ and it follows from (a) that $n_x(\a + \b) = n_x(\a) = n_x(\b)$. Assume further that $0 < \th < \pi$. Then $\dim U = 2$ and $U \setminus (L_\a \cup L_\b)$ has four connected components, which we label by $D$, $D'$, $-D$ and $-D'$. Assume that $e \in D$. Then every point $v\in D$ must satisfies $\<v,\a\>>0$ and $\<v,\b\>>0$. Hence we have

    (b) $L_{\a + \b}$ intersects $D'$ and $-D'$. 

    Suppose that $x^i(e)$ and $e$ are separated by $L_{\a+\b}$ for some $i \in \BZ$, then (b) implies that they are also separated by $L_\a$ or $L_\b$. Using (a), we deduce that $n_x(\a + \b) \le \max\{n_x(\a), n_x(\b)\}$. Hence (3) is proved.

    Note that $V_{x}^\th = V_{x\i}^\th$ and $\Phi(x) = \Phi(x\i)$ by definition. Then (4) follows from (1), (2), (3).
    \end{proof}

For any subset $R \subseteq \Phi$ and $y\in W\rtimes\<\d\>$, we write $R^y = \{\a \in R; y(\a) = \a\}$. The main result of this section is the following.
\begin{theorem}\label{thm:exist}
    For any $x \in W \rtimes \<\d\>$, there exists a $W$-conjugate $y$ of $x$ such that $y$ is convex (for $\Phi^+$) and $\Phi(y) = \Phi^y$.
\end{theorem}
\begin{proof}
    We argue by induction on $\sharp W$. If $W = \{1\}$ or $V = V_x^0$, then $\Phi = \Phi^x$ and the statement is trivial. Otherwise, there exists $0 < \th \le \pi$ such that $V_x^\th \neq \{0\}$. Up to $W$-conjugation, we can assume that $\overline{C_0}$ contains a regular point $e$ of $V_x^\th$. Note that $W_e$ is a standard proper subgroup of $W$. By induction hypothesis, up to $W_e$-conjugation we can assume further that $\Phi_e(x) = (\Phi_e)^x$ and $x$ is convex for $\Phi_e^+: = \Phi_e \cap \Phi^+$. By Lemma \ref{lem:rotation} (2) and (4), we deduce that $\Phi(x) = \Phi^x$ and $x$ is convex for $\Phi^+$ as desired.
\end{proof}

\section{Good position elements} \label{sec:good}
Good position elements in $W$ are first introduced in \cite{HL}. They play a crucial role in the construction of transversal slices of unipotent orbits by He-Lusztig \cite{HL} and Duan \cite{Duan}.  

In this section, we show that good position elements are convex. As an application, we give a new proof of a classical result by He and Lusztig that $\Xi_\dx$ is an isomorphism for any elliptic minimal length element $x$. Finally, we provide a counterexample showing that the set of convex elements is strictly larger than the set of good position elements.

\subsection{Good position elements}
Let $x \in W \rtimes \<\d\>$. We say a sequence $\underline{\th} = (\th_1,\th_2,\ldots,\th_r)$ in $(0, \pi]$ is $\Phi$-admissible for $x$ if $\sum_{j=1}^r V_x^{\th_j}$ contains a $\Phi$-regular point of $(V^x)^\perp = \bigoplus_{0 < \th \le \pi} V_x^\th$.
Note that any enumeration of the set $\{\th\in (0,\pi] ; V_x^{\th} \ne 0\}$ is a $\Phi$-admissible sequence for $x$. By convention, the empty sequence is $\Phi$-admissible for $x$ if and only if $x=1$.

\begin{definition}[{\cite[\S5.2]{HN12}, \cite[\S2.1]{Duan}}]\label{def:good}
Let $x \in W \rtimes\<\d\>$ and $\underline{\th} = (\th_1,\th_2,\ldots,\th_r)$ be a $\Phi$-admissible sequence for $x$. We say $x$ is \emph{at good position} with respect to $(\Phi^+,\underline{\th})$ if $\overline{C_0}$ contains some $\Phi$-regular point of $\sum_{j=1}^{i} V_x^{\th_j}$ for each $i$. \footnote{Our definition is slightly different from that in \cite{HN12} and \cite{Duan}.}
\end{definition}

\begin{example}\label{ex:A3}
In the case of type $A_3$, let $x = s_2s_1s_3$. We have $V = \{(a,b,c,d)\in \BR^4; a+b+c+d=0\} = V_{x}^{\pi/2}\oplus V_{x}^{\pi}$, $V_{x}^{\pi} = \{(a,-a,-a,a); a\in\BR\}$ and $V_{x}^{\pi/2} = \{(a,b,-b,-a); a,b\in\BR\}$. Then $x$ is at good position with respect to the sequence $(\pi/2,\pi)$ but is not at good position with respect to the sequence $(\pi,\pi/2)$. Let $x' = s_1s_2s_3$, one can check that $x'$ is not at good position with respect to $(\pi/2,\pi)$ or $(\pi,\pi/2)$. 
\end{example}

The following lemma provides a recursive definition for good position elements. The proof is straightforward by definition.
\begin{lemma} \label{lem:inductive}
    Let $x \in W \rtimes\<\d\>$ and let $\underline{\th} = (\th_1,\th_2,\ldots,\th_r)$ be an $\Phi$-admissible sequence for $x$. Then $x$ is at good position with respect to $(\Phi^+, \underline{\th})$ if and only if $\overline{C_0}$ contains a regular point of $V_x^{\th_1}$ and $x$ is at good position with respect to $(\Phi_1^+, \underline{\th}_{\ge 2})$, where $\Phi_1 = \{\g \in \Phi; V_x^{\th_1} \subseteq H_\g\}$ and $\underline{\th}_{\ge 2} = (\th_2, \th_3, \dots, \th_r)$.
\end{lemma}

The main result for this section is the following.
\begin{theorem} \label{thm:constr}
Let $x \in W\rtimes\<\d\>$ and $\underline{\theta}=(\th_1,\ldots,\th_r)$ be a $\Phi$-admissible sequence for $x$. If $x$ is at good position with respect to $(\Phi^+, \underline{\th})$, then $x$ is convex for $\Phi^+$ and $\Phi(x) = \Phi^{x}$.
\end{theorem}
\begin{proof}
    We argue by induction on $r$. We first assume that $r = 0$. Then $x = 1$ and the statement is trivial. Now assume $r>0$. Suppose that the statement holds for all $r' < r$. Let notation be as in Lemma \ref{lem:inductive}. Then $\overline{C_0}$ contains a $\Phi$-regular point of $V_x^{\th_1}$ and $x$ is at good position with respect to $(\Phi_1^+, \underline{\th}_{\ge 2})$. By induction hypothesis, $x$ is convex for $\Phi_1^+$ and $\Phi_1(x) = (\Phi_1)^x$. By Lemma \ref{lem:rotation} (2), (4), $x$ is convex for $\Phi^+$ and $\Phi(x) = \Phi^x$ as desired.   
\end{proof}

\begin{remark} \label{rmk:compare}
\begin{enumerate}
\item The statement $\Phi(x) = \Phi^{x}$ in Theorem \ref{thm:constr} implies that if $x$ is at good position, then the variety $\CS_\dx$ (see \S\ref{subsec:cross-section}) coincides with the variety $S_{\mathrm{Br}}^D(\tilde{b})$ in \cite[Definition 4.3]{Duan}.
\item Let $x\in W\rtimes \<\d\>$ and $\underline{\th}$ be an admissible sequence for $x$. By {\cite[Lemma 5.1]{HN12}}, there exists a $W$-conjugate $y$ of $x$ that is at good position with respect to $(\Phi^+, \underline{\th})$. Then Theorem \ref{thm:constr} implies that $y$ is convex and $\Phi(y) = \Phi^y$. This gives essentially the same construction as in the proof of Theorem \ref{thm:exist}.

\end{enumerate}
\end{remark}

\subsection{A result by He-Lusztig}
In this subsection, we provide a new and simpler proof of the main result of \cite{HL} by He and Lusztig.

Recall that $x \in W\rtimes \<\d\>$ is called elliptic if $x$ does not fix any non-zero vector in $V$.
\begin{lemma}\label{lem:elliptic}
Let $x \in W \rtimes \<\d\>$ be elliptic. Then $\Phi(x) = \emptyset$.    
\begin{proof}
Otherwise, we may choose $\g \in \Phi^+(x)$ such that $x^i(\g) \in \Phi^+$ for all $i \in \BZ$. Let $k \in \BZ_{\ge 1}$ such that $x^k = \id$. Then the sum $\sum_{i=0}^{k-1} x^i(\g)$ of $x$-orbit of $\g$ is a non-zero vector fixed by $x$, a contradiction. 
\end{proof}
\end{lemma}

\begin{definition}\label{def:cyclic-shift}
Let $x, x' \in W \rtimes \<\d\>$ and $s \in \BS$, We write $x \rightarrow x'$ if there exist elements $x = x_1, x_2, \dots, x_k=x'$ in $W \rtimes \<\d\>$ and simple reflections $s_1,s_2,\ldots, s_{k-1}$ such that  $x_{i+1} = s_i x_i s_i$ and $\ell(x_{i+1}) \le \ell(x_i)$ for $1 \le i \le k-1$. We $x \approx x'$ if $x \rightarrow x'$ and $x' \rightarrow x$. In this case, we also say $x$ is a cyclic shift of $x'$.   
\end{definition}

For any $W$-conjugacy class $\CO$ in $W \rtimes \<\d\>$, we define $\CO_{\min} = \{y \in \CO; \ell(y) \le \ell(z), \forall z \in \CO\}$. An element $x$ is called a minimal length element if $x \in \CO_{\min}$ for some $W$-conjugacy class $\CO$.
\begin{proposition}\label{prop:elliptic}
Let $\CO$ be an elliptic $W$-conjugacy class. Then for any $x \in \CO$ there exists $x' \in \CO_{\min}$ such that $x'$ is convex and $x \rightarrow x'$.
\end{proposition}
\begin{proof}
Let $\underline{\th} = (\th_1,\th_2,\ldots,\th_r)$ be the ascending enumeration of $\{\th\in (0,\pi]; V_{x}^{\th}\ne0\}$. Thanks to \cite[Proof of Theorem 3.2(1)]{HN12}, there exists $x' \in \CO_{\min}$ such that $x \rightarrow x'$ and $x'$ is at good position with respect to $(\Phi^+, \underline{\th})$. Then Theorem \ref{thm:constr} implies that $x'$ is convex as desired. 
\end{proof}
\begin{remark}
\begin{enumerate}
    \item The condition that $\CO$ is elliptic in Proposition \ref{prop:elliptic} is necessary. For example, in the case of type $A_2$. Consider the conjugacy class $\CO = \{s_1,s_2,s_1s_2s_1\}$. There is no convex element in $\CO_{\min} = \{s_1,s_2\}$. Indeed, $\Phi(s_1) = \{\pm\a_2,\pm\a_1\pm\a_2\}$ is not a standard parabolic root subsystem (and similar for $s_2$). However, the element $s_1s_2s_1$ is at good position with respect to the sequence $\underline{\th} = (\pi)$ and is convex.
    \item In the setting of Proposition \ref{prop:elliptic}, not all elements in $\CO_{\min}$ are convex. As an example, in the case of type $C_3$, the element $x =  s_3s_2s_3s_1s_2$ is elliptic and is of minimal length in its conjugacy class but is not convex. 
\end{enumerate}
\end{remark}

We can now prove the following classical result by He and Lusztig.
\begin{theorem}[{\cite[Theorem 3.5]{HL}}]\label{thm:HL} Let $x$ be an elliptic minimal length element. Then the map
\begin{align*}
U \times \dx U_{\Phi^+\cap x\i(\Phi^-)} \longrightarrow U \dx U, \quad (y,z) \mapsto  y z y\i
\end{align*}
is an isomorphism.
\end{theorem}
\begin{proof}
By Proposition \ref{prop:elliptic}, there exists a convex element $x'$ such that $x \approx  x'$. By \cite[\S3.5 (a), (b)]{HL}, $\Xi_\dx$ is an isomorphism if and only if so is $\Xi_{\dx'}$. Hence we can and do assume that $x$ is convex. Note that $\Phi(x) = \emptyset$ by Lemma \ref{lem:elliptic}. Then the statement follows from Theorem \ref{thm:main}.
\end{proof}

\begin{example}\label{rmk:non-convex}
In \cite[\S0.4]{HL}, He and Lusztig show that the map $\Xi_\dx$ for $x = (1, 6, 4, 5, 2, 3)$ in the case of $G = \GL_6$ is not injective. In fact, one can check that $n_x(\a_2 ) = 1$, $n_x( \a_3) = 2$, $n_x( \a_2+\a_3) = 3$ and hence $x$ is not quasi-convex.
\end{example}

\subsection{A counterexample}
In this subsection, we show that set of convex elements is strictly larger than the set of good position elements. Thus Theorem \ref{thm:main} provides many new cases of Steinberg cross-sections.

We first recall a length formula for good position elements.
\begin{lemma}[{\cite[Lemma 2.1]{HN12}, \cite[Corollary 3.11]{Duan}}]\label{lem:length}
Let $x\in W\rtimes\<\d\>$ be a good position element with respect to $(\Phi^+,\underline{\th} =(\th_1,\th_2\ldots,\th_r))$. Then
$$\ell(x) = \sum_{i=1}^{r}\frac{\th_i}{\pi} \bigl(h_{i-1}-h_i\bigr) ,$$
where $h_0 = \sharp \Phi^+$ and $h_i = \sharp\{\g\in\Phi^+ ; H_{\g}\supseteq \bigoplus_{j=1}^{i} V_{x}^{\th_{j}}  \}$ for $1 \le i \le r$.  
\end{lemma}
We now explicitly construct a convex element which is not a cyclic shift of any good position element with respect to any admissible sequence.
\begin{example}\label{ex:234123}
In the case of type $A_4$. There are only two $\Phi$-admissible sequences for the conjugacy class the Coxeter element $s_1 s_2 s_3 s_4$, which are $\underline{\th} = (\frac{2\pi}{5},\frac{4\pi}{5})$ and $\underline{\th'} = (\frac{4\pi}{5},\frac{2\pi}{5})$. By Lemma \ref{lem:length}, we deduce that all good position elements with respect to $\underline{\th}$ have length $\frac{(2\pi/5)}{\pi}10 = 4$, and all good positions elements with respect to $\underline{\th'}$ have length $\frac{(4\pi/5)}{\pi}10 = 8$. However, we can check directly that $x = s_2s_3s_4s_1s_2s_3 \in \CO$ is convex. But $x$ has length $6$ and hence is not a cyclic shift of any good position element with respect to any admissible sequence.
\end{example}

\section{Coxeter elements}\label{sec:Coxeter}


An element $c$ of $W$ is said to be a $\d$-Coxeter element, if $c$ is the product of some simple reflections, one from each $\d$-orbit of $\BS$. It is natural to ask if all $\d$-Coxeter elements are convex. We propose the following conjecture.
\begin{conjecture}\label{conj:cox}
Let $c$ be a $\d$-Coxeter element of $W$. Then $c\d \in W\rtimes\<\d\>$ is convex. 
\end{conjecture}

Let $h$ be order of $c\d$ for some/any $\d$-Coxeter element. Let $w_0 \in W$ denote the longest element. 
\begin{proposition}\label{prop:cox-convex}
Conjecture \ref{conj:cox} is true if $h$ is even and $(c\d)^{\frac{h}{2}} = w_0\d^{\frac{h}{2}}\in W\rtimes\<\d\>$. 
\end{proposition}

\begin{remark}Suppose the Dynkin diagram is connected. We illustrate some cases where the condition in Proposition \ref{prop:cox-convex} holds.

\begin{enumerate}
    \item If $\d = \id$ and $W$ is not of type $A_n$ ($n$ even), then there exists some Coxeter element $c$ such that $c^{h/2} = w_0$ (cf. \cite[Lemma 3.2 (i)]{Lusz76}).
    \item If $\d = \id$ and $w_0 = -1$ (this is true except in the case of type $A_n$, $D_n$ ($n$ odd), and $E_6$), then the condition $c^{h/2} = w_0$ holds for all Coxeter element $c$ (cf. \cite[Corollary 3.19]{Hum90}).
    \item If $\d \ne \id$, then the condition $(c\d)^{\frac{h}{2}} = w_0\d^{\frac{h}{2}}$ holds for all $\d$-Coxeter element $c$ (cf. \cite[Lemma 3.2 (ii)]{Lusz76}).
\end{enumerate}

\end{remark}

In order to prove Proposition \ref{prop:cox-convex}, we recall the notion of reflection ordering (see for example \cite[\S2.1]{BFP}). A reflection ordering on $\Phi^+$ is a total order $<$ such that for any $\a,\b\in\Phi^+$ with $\a + \b\in \Phi^+$, we have either $\a< \a + \b < \b$ or $\b < \a + \b < \a$. 

Let $N = \ell(w_0)=\sharp \Phi^+$. By \cite[Proposition 2.13]{Dyer93}, reflection orderings of $\Phi^+$ are in bijective correspondence with reduced expression of $w_0$, namely, $\b_N<\b_{N-1}<\cdots<\b_1$ is a reflection ordering if and only if there exists a reduced expression $w_0 = s_1s_2\cdots s_N$ such that $\b_i = s_N s_{N-1}\cdots s_{i+1}\a_i$ for $i=1,2,\ldots, N$. Here, $\a_i$ is the simple root corresponding to $s_i$. 


\begin{proof}[Proof of Proposition \ref{prop:cox-convex}]
Let $c = s_1s_2\cdots s_n$ be a reduced expression of $c$, where $n=\sharp\D$. Let $\a_i$ be the simple root corresponding to $s_i$ for each $i$. By assumption, $$w_0 = s_1s_2\cdots s_n \d(s_1)\d(s_2)\cdots \d(s_n)\cdots \d^{\frac{h}{2}-1}(s_1)\d^{\frac{h}{2}-1}(s_2)\cdots \d^{\frac{h}{2}-1}(s_n)$$ is a reduced expression. Set $\b_i = s_n s_{n-1}\cdots s_{i+1}\a_i$ for $i=1,2,\ldots, n$. Applying the bijection mentioned above, this reduced expression of $w_0$ gives rise to a reflection ordering
\begin{align*}\tag{4.1}\label{4.1}
    &(c\d)^{\frac{h}{2}-1}\b_n < (c\d)^{\frac{h}{2}-1}\b_{n-1}< \cdots <(c\d)^{\frac{h}{2}-1}\b_1\\
    <&(c\d)^{\frac{h}{2}-2}\b_n < (c\d)^{\frac{h}{2}-2}\b_{n-1}< \cdots < (c\d)^{\frac{h}{2}-2}\b_1\\
    <&\cdots\\
    <&(c\d)\b_n < (c\d)\b_{n-1}<\cdots < (c\d)\b_1\\
    <&\b_n <  \b_{N-1} <\cdots< \b_1 .
\end{align*}
Hence, for $i =1,2,\ldots,h/2$, we have 
\begin{align*}\tag{4.2}\label{4.2}
\{\g \in \Phi^+;  n_{c\d}(\g) = i\} = \{ (c\d)^{\frac{h}{2}-i}\b_n, (c\d)^{\frac{h}{2}-i}\b_{n-1} ,\ldots,  (c\d)^{\frac{h}{2}-i}\b_1  \};\\
\{\g \in \Phi^+;  n_{(c\d)\i}(\g) = i\} = \{ (c\d)^{i-1}\b_n, (c\d)^{i-1}\b_{n-1} ,\ldots,  (c\d)^{i-1}\b_1  \}.
\end{align*}

Let $\g,\g'\in\Phi^+$ such that $\g+\g'\in \Phi^+$. By the definition of reflection ordering, the root $\g+\g'$ lies between $\g$ and $\g'$ in the sequence (\ref{4.1}). Then $n_{(c\d)^{\pm1}}(\g+\g')\le  \max\{n_{(c\d)^{\pm1}}(\g),n_{(c\d)^{\pm1}}(\g')\}$ by (\ref{4.2}). By Lemma \ref{lem:elliptic}, $\Phi(c\d) = \emptyset$. This proves that $c\d$ is convex.
\end{proof}






\printbibliography

\end{document}